\documentclass[12pt,twoside]{amsart}
\usepackage{amssymb}
\usepackage{amscd}
\usepackage{amsmath}
\usepackage{xypic}

\title{On log surfaces}
\author{Osamu Fujino} 
\author{Hiromu Tanaka}
\date{2012/5/11}
\address{Department of Mathematics, Faculty of Science, 
Kyoto University, Kyoto 606-8502, Japan}
\email{fujino@math.kyoto-u.ac.jp}
\address{Department of Mathematics, Faculty of Science, 
Kyoto University, Kyoto 606-8502, Japan}
\email{tanakahi@math.kyoto-u.ac.jp}

\newcommand{\Pic}[0]{{\operatorname{Pic}}}

\newcommand{\Supp}[0]{{\operatorname{Supp}}}
\newcommand{\Spec}[0]{{\operatorname{Spec}}}
\newcommand{\Bs}[0]{{\operatorname{Bs}}}
\newcommand{\ch}[0]{{\operatorname{char}}}

\newtheorem{thm}{Theorem}[section]

\newtheorem{cor}[thm]{Corollary}
\newtheorem{prop}[thm]{Proposition}
\newtheorem*{claim}{Claim}

\theoremstyle{definition}
\newtheorem{ex}[thm]{Example}
\newtheorem{defn}[thm]{Definition}
\newtheorem{rem}[thm]{Remark}
\newtheorem*{ack}{Acknowledgments}       
\newtheorem*{notation}{Notation}

\begin{document}

\maketitle 

\begin{abstract}
This paper is an announcement of the 
minimal model theory for log surfaces in all characteristics 
and contains some related results 
including a simplified proof of the Artin--Keel contraction theorem in the surface case. 
\end{abstract} 

\tableofcontents

\section{Introduction}\label{sec0} 

In this paper, we will work over an algebraically closed field $k$ of characteristic zero or 
positive characteristic. 
Let us recall the definition of {\em{log surfaces}}. 

\begin{defn}[Log surfaces] 
Let $X$ be a normal algebraic surface and let $\Delta$ be a boundary $\mathbb R$-divisor on
$X$ such that $K_X+\Delta$ is $\mathbb R$-Cartier. 
Then the pair $(X, \Delta)$ is called a {\em{log surface}}. 
We recall that  
a {\em{boundary}} $\mathbb R$-divisor is an effective $\mathbb R$-divisor whose 
coefficients are less than or equal to one. 
\end{defn}

In the preprint \cite{tanaka} together with
the forthcoming \cite{fujino},
we have obtained the following theorem, that is, 
the minimal model theory for log surfaces. 

\begin{thm}[{cf.~\cite{fujino} and \cite{tanaka}}]\label{main-ka}
Let $\pi:X\to S$ be a projective morphism from a log surface $(X, \Delta)$ to 
a variety $S$ over a field $k$ of arbitrary characteristic. 
Assume that one of the following conditions holds{\em{:}} 
\begin{itemize}
\item[(A)] $X$ is $\mathbb Q$-factorial, or 
\item[(B)] $(X, \Delta)$ is log canonical. 
\end{itemize}
Then we can run the log minimal model program over $S$ 
with respect to $K_X+\Delta$ and 
obtain a sequence of extremal contractions  
\begin{align*}
(X, \Delta)=(X_0, \Delta_0)
\overset{\varphi_0}\to (X_1, \Delta_1)\overset{\varphi_1}\to 
\cdots \\ 
\overset{\varphi_{k-1}}\to (X_k, \Delta_k)=(X^*, \Delta^*)
\end{align*} 
over $S$ 
such that 
\begin{itemize}
\item[(1)] {\em{(Minimal model)}} $K_{X^*}+\Delta^*$ is semi-ample over $S$ 
if $K_X+\Delta$ is pseudo-effective over $S$, or
\item[(2)] {\em{(Mori fiber space)}} there is a morphism 
$g:X^*\to C$ over $S$ such that 
$-(K_{X^*}+\Delta^*)$ is $g$-ample, $\dim C<2$, and 
the relative Picard number $\rho (X^*/C)=1$, 
if $K_X+\Delta$ is not pseudo-effective over $S$. 
\end{itemize}
We note that, in {\em{Case (A)}}, we do not assume that 
 $(X, \Delta)$ is log canonical. We also note that 
$X_i$ is $\mathbb Q$-factorial for every $i$ in {\em{Case (A)}} 
and that $(X_i, \Delta_i)$ is log canonical for every $i$ in {\em{Case (B)}}. 
Moreover, in both cases, if $X$ has only rational singularities, 
then so does $X_i$ by 
{\em{Theorem \ref{thm-b}}}. 
\end{thm}

More precisely, we prove the cone theorem, the contraction theorem, and 
the abundance theorem for 
$\mathbb Q$-factorial log surfaces and log canonical surfaces 
with no further restrictions. 

Theorem \ref{main-ka} has been known 
in Case (B) when $S$ is a point and $\Delta$ is a $\mathbb Q$-divisor (cf.~\cite{fujita3}, \cite{koko}).
In \cite{fujino}, the first author obtained Theorem \ref{main-ka} in characteristic zero;
there are many $\mathbb Q$-factorial surfaces (i.e.~in Case (A)) which are not log canonical (i.e.~in Case (B)). 

In \cite{tanaka}, the second author establishes Theorem \ref{main-ka} in arbitrary positive characteristic. 
The arguments in \cite{fujino} heavily depend on a Kodaira type vanishing theorem, which
unfortunately fails in positive characteristic. 
The main part of discussion in \cite{tanaka}
in order to prove 
Theorem \ref{main-ka} 
is the Artin--Keel contraction theorem, which holds only in positive characteristic. 

We will give a simplified proof of the Artin--Keel contraction theorem in Section \ref{sec1}, which is one of the main 
purposes of this paper.  
 
Theorem \ref{main-ka} implies the following important corollary. 
For a more direct approach to Corollary \ref{cor13}, see Section \ref{sec4}.  

\begin{cor}\label{cor13}
Let $(X, \Delta)$ be a projective 
log surface. 
Assume that 
$X$ is $\mathbb Q$-factorial or $(X, \Delta)$ is log canonical.  
Then the log canonical ring 
$$
R(X, \Delta)=\bigoplus _{m \geq  0}H^0(X, \mathcal O_X(\llcorner m(K_X+\Delta)\lrcorner))
$$ 
is a finitely generated $k$-algebra. 
\end{cor}

\begin{rem}Let $(X, \Delta)$ be a log canonical surface and let $f:Y\to X$ be 
its minimal resolution with $K_Y+\Delta_Y=f^*(K_X+\Delta)$. 
Then $(Y, \Delta_Y)$ is a $\mathbb Q$-factorial log surface. 
The abundance theorem and the finite generation of log canonical rings 
for $(X, \Delta)$ follow from those of $(Y, \Delta_Y)$. 
\end{rem}

In Section \ref{sec-5}, we discuss a key result on the {\em{indecomposable 
curves of canonical type}} for the proof of the abundance theorem for 
$\kappa =0$. 
The behavior of the indecomposable curves of canonical 
type varies in accordance with the cases:~(i) 
$\ch (k)=0$, (ii) $\ch (k)>0$ with 
$k\ne \overline {\mathbb F}_p$, and 
(iii) $k=\overline {\mathbb F}_p$. 
Section \ref{sec2} is devoted to the discussion of some relative vanishing theorems, which 
are elementary and hold in any characteristic. 
As applications, 
we give some results supplementary 
to the theory of algebraic surfaces in {\em{arbitrary}} characteristic. 

For the details of the proofs, we refer
the reader to \cite{fujino} and \cite{tanaka}. 

\begin{notation} 
For an $\mathbb R$-divisor $D$ on a normal surface $X$, we define the 
{\em{ronud-up}} $\ulcorner D\urcorner$, the {\em{round-down}} $\llcorner 
D\lrcorner$, and the {\em{fractional part}} $\{D\}$ of $D$. 
We note that $\sim _{\mathbb R}$ denotes the {\em{$\mathbb R$-linear equivalence}} of 
$\mathbb R$-divisors. 
Let $D$ be a $\mathbb Q$-Cartier $\mathbb Q$-divisor on a normal 
projective surface $X$. 
Then $\kappa (X, D)$ denotes the {\em{Iitaka--Kodaira dimension}} of $D$. 
Let $k$ be a field. Then $\ch (k)$ denotes the characteristic of 
$k$. Let $A$ be an abelian variety defined over an algebraically closed field 
$k$. Then we denote the $k$-rational points of $A$ by $A(k)$. 
\end{notation}

\begin{ack}
The first author was partially supported by the Grant-in-Aid for Young Scientists 
(A) $\sharp$20684001 from JSPS. 
He was also supported by the Inamori Foundation. 
He thanks Professors Fumio Sakai and Se\'an Keel for comments on \cite{fujino}.  
The authors would like to thank Professors Kenji Matsuki and Shigefumi Mori for 
stimulating discussions and useful comments, and 
Professor Atsushi Moriwaki for warm encouragement. 
\end{ack}

\section{the Artin--Keel contraction theorem}\label{sec1}

The following (see Theorem \ref{main-thm}) is Keel's base point free theorem 
for algebraic surfaces (cf.~\cite[0.2 Theorem]{keel}). 
Although his original result holds in {\em{any}} dimension, we only discuss 
it for surfaces here.  
The paper \cite{keel} attributes Theorem \ref{main-thm} to \cite{artin} 
even though it is not stated explicitly there. 
So, we call it the Artin--Keel contraction theorem in this paper. 
Theorem \ref{main-thm} will play crucial roles in 
the 
minimal model theory 
for log surfaces in positive characteristic. 
For the details, see \cite{tanaka}. 
Note that Theorem \ref{main-thm} fails in characteristic zero 
by \cite[3.0 Theorem]{keel}. 
The minimal model theory for 
log surfaces in characteristic zero 
heavily depends on a Kodaira type vanishing theorem (cf.~\cite{fujino}). 
The second author discusses the X-method for klt surfaces in positive 
characteristic in \cite{tanaka2}. 
For a related topic, see also \cite{cmm}. 

\begin{thm}[Artin, Keel]\label{main-thm} 
Let $X$ be a complete normal algebraic surface defined over an 
algebraically closed filed $k$ of positive characteristic and 
let $H$ be a nef and big Cartier divisor on $X$. 
We set  
$$
\mathcal E(H):=\{C\, |\, C\ \text{is a curve on $X$ and}\ C\cdot H=0\}. 
$$
Then $\mathcal E(H)$ consists of finitely many irreducible curves on $X$. 
Assume that $H|_{\mathbb E(H)}$ is semi-ample where 
$$\mathbb E(H)=\bigcup _{C\in \mathcal E(H)}C$$ with the 
reduced scheme structure. 
Then $H$ is semi-ample. Therefore, 
$$\Phi_{|mH|}:X\to Y$$ is a proper birational 
morphism onto a normal projective surface $Y$ which contracts 
$\mathbb E(H)$ and is an isomorphism outside $\mathbb E(H)$ for 
a sufficiently large and divisible positive integer. 
\end{thm}
We give two different proofs of Theorem \ref{main-thm}. 
Proof 1 depends on Artin's arguments. 
On the other hand, Proof 2 uses Fujita's vanishing theorem. 

\begin{proof}[{\em{Proof $1$}}] 
It is sufficient to prove that $H$ is semi-ample. 
Let $f:Z\to X$ be a resolution of singularities. Then $\mathcal E(f^*H)$ consists of 
finitely many curves by the Hodge index theorem. Therefore, 
so does $\mathcal E(H)$. 
Note that $H$ is semi-ample if and only if $f^*H$ is semi-ample. 
We also note that 
$f^*H|_{\mathbb E(f^*H)}$ is semi-ample since so is $H|_{\mathbb E(H)}$. 
Thus, by replacing $X$ and $H$ with $Z$ and $f^*H$, we may assume that 
$X$ is a smooth projective surface. 
In this case, the intersection matrix of $\mathcal E(H)$ is negative definite 
by the Hodge index theorem. 

By Artin's contraction theorem (see \cite[Theorem 14.20]{bad}), 
there exists a morphism $g:X\to W$ where 
$W$ is a normal complete two-dimensional algebraic space such that 
$g(\mathbb E(H))$ is a finite set of points of $W$ and that 
$$g|_{X\setminus \mathbb E(H)}:X\setminus \mathbb E(H)\to W\setminus 
g(\mathbb E(H))$$ is an isomorphism. 

By Artin (cf.~\cite[Lemma (2.10)]{artin} and \cite[Step 1 in the proof 
of Theorem 14.21]{bad}), 
there exists an effective Cartier divisor $E$ with $\Supp (E)=\mathbb E(H)$ such that 
for every effective divisor 
$D\geq E$ with $\Supp D=\mathbb E(H)$, the restriction map $\Pic (D)\to \Pic (E)$ is an 
isomorphism. 

By replacing $H$ with a multiple, 
we may assume that $H|_{\mathbb E(H)}$ is free. 
Therefore, $\mathcal O_{\mathbb E(H)}(H)\simeq \mathcal O_{\mathbb E(H)}$. 

Let $p$ be the characteristic of $k$ and 
let $r$ be a positive integer such that 
$q\mathbb E(H)\geq E$ where $q=p^r$. We consider the (ordinary) $q$-th 
Frobenius morphism $F:X\to X$. 
By pulling back the exact sequence 
$$
0\to \mathcal O_X(H-\mathbb E(H))\to \mathcal O_X(H)\to \mathcal O_{\mathbb E(H)}\to 0
$$
by $F$, we obtain the exact sequence 
$$
0\to \mathcal O_X(qH-q\mathbb E(H))\to \mathcal O_X(qH)\to \mathcal O_{q\mathbb E(H)}
(qH)\to 0. 
$$
Therefore, $\mathcal O_{q\mathbb E(H)}(qH)\simeq 
\mathcal O_{q\mathbb E(H)}$. By the above argument, 
$\mathcal O_D(qH)\simeq \mathcal O_D$ for every 
effective divisor $D\geq E$ with $\Supp D=\mathbb E(H)$.  

Let $w\in g(\mathbb E(H))$ be a point. 
Then, by the theorem of holomorphic functions (cf.~\cite[Theorem 3.1]{knu}), 
we have 
$$(g_*\mathcal O_X(qH))^{\wedge}_w\simeq (g_*\mathcal O_X)^{\wedge}_w\simeq 
\widehat{\mathcal O_{W, w}}. $$
Therefore, $g_*\mathcal O_X(qH)$ is a line bundle 
on $W$. 
By considering the natural map
$$
g^*g_*\mathcal O_X(qH)\to \mathcal O_X(qH), 
$$ 
we obtain that $g^*g_*\mathcal O_X(qH)\simeq \mathcal O_X(qH)$ since 
the intersection matrix of $\mathcal E(H)$ is 
negative definite. 
This means that 
$B:=g_*qH$ is a Cartier divisor on $W$ and $g^*B=qH$. By Nakai's 
criterion, $B$ is ample on $W$. We note that 
Nakai's criterion holds for complete algebraic spaces. 
Therefore, $H$ is semi-ample. 
\end{proof}

\begin{proof}[{\em{Proof $2$}}] 
It is sufficient to prove that $H$ is semi-ample. 
By the same argument as in Proof $1$, we may assume that 
$X$ is smooth. 
We may further assume that $\mathcal O_
{\mathbb E(H)}(H)\simeq \mathcal O_{\mathbb E(H)}$ by replacing $H$ with a multiple. 
Since the intersection matrix of $\mathcal E(H)$ is negative 
definite, we can find an effective Cartier divisor $D$ such that 
$\Supp D=\mathbb E(H)$ and that 
$-D\cdot C>0$ for every $C\in \mathcal E(H)$. 
\begin{claim}
There exists a positive integer $m$ such that $mH-D$ is ample. 
\end{claim}
\begin{proof}[Proof of {\em{Claim}}]
By Kodaira's lemma, we can find a positive integer $k$, 
an ample Cartier divisor $A$, and an effective divisor $B$ 
such that $kH-D\sim A+B$. 
If $(kH-D)\cdot C\leq 0$ for some curve $C$, then $C\subset \Supp B$ and 
$C\not\in \mathcal E(H)$. Therefore, if $m\gg k$, then $(mH-D)\cdot C>0$ for 
every curve $C$ on $X$. This implies $mH-D$ is ample since $mH-D$ is a big divisor. 
\end{proof}

By replacing $mH-D$ with a multiple, we may assume that 
$mH-D$ is very ample and 
$H^1(X, \mathcal O_X(lH-D))=0$ for every $l\geq m$ by 
Fujita's vanishing theorem (see \cite[Theorem (1)]{fujita-lnm} and \cite[(5.1) Theorem]{fujita2}). 
Let $p$ be the characteristic of $k$ and let $r$ be a positive integer 
such that $q\mathbb E(H)\geq D$ where $q=p^r$. 
By the same argument as in Proof $1$, 
$\mathcal O_{q\mathbb E(H)}(qH)\simeq \mathcal O_{q\mathbb E(H)}$. Therefore, 
$qH|_D\sim 0$. 
Without loss of generality, we may further assume that 
$q\geq m$. 
By the exact sequence 
$$
0\to H^0(X, \mathcal O_X(qH-D))\to H^0(X, \mathcal O_X(qH))\to 
H^0(D, \mathcal O_D(qH))\to 0, 
$$ 
$\Bs |qH|\cap \mathbb E(H)=\emptyset$ where 
$\Bs |qH|$ is the base locus of the linear system 
$|qH|$. 
Since $mH-D$ is ample with $\Supp D=\mathbb E(H)$, we obtain that 
$H$ is semi-ample. 
\end{proof}

\begin{cor}\label{cor22} Let 
$X$ be a $\mathbb Q$-factorial 
projective surface defined over an algebraically closed 
field of positive characteristic. 
Let $C$ be a curve on $X$ such that 
$C\simeq \mathbb P^1$ and $C^2<0$. 
Then we can contract $C$ to a $\mathbb Q$-factorial point. 
\end{cor}

\begin{proof}[Sketch of the proof] 
Let $H$ be a very ample Cartier divisor 
on $X$. 
We set $L=(-C^2)H+(H\cdot C)C$. 
Then $L$ is nef and big. 
Note that $L|_C$ is semi-ample since 
$C\simeq \mathbb P^1$ and $L\cdot C=0$. 
By applying Theorem \ref{main-thm} to $L$, 
we have a desired contraction morphism. 
\end{proof}

Since $\Pic ^0(V)(k)$ is a torsion group for any projective variety 
$V$ defined over $k=\overline {\mathbb F}_p$, 
we obtain the following corollary. 

\begin{cor}[{cf.~\cite[0.3 Corollary]{keel}}]\label{cor23} 
Let $X$ be a normal projective surface over $k=\overline {\mathbb F}_p$ 
and let $D$ be a nef and big Cartier divisor on $X$. 
Then $D$ is semi-ample. 
\end{cor}

In \cite[Section 3]{keel}, Keel 
obtained an interesting example. 

\begin{prop}[{cf.~\cite[3.0 Theorem]{keel}}]\label{prop-keel} 
Let $C$ be a smooth projective curve of genus $g\geq 2$ over an 
algebraically closed filed $k$. 
We consider $S=C\times C$. 
We set $L=\pi_1^*K_C+\Delta$ where 
$\Delta\subset S$ is the diagonal and $\pi_1:S\to C$ is the 
first projection. 
Then $L$ is semi-ample if and only if 
the characteristic of $k$ is positive. 
\end{prop}

By Proposition \ref{prop-keel}, 
we obtain the following interesting example. 

\begin{ex}
Let $U$ be a nonempty Zariski open set of 
$\Spec \mathbb Z$ and let $X\to U$ be 
a smooth family of curves of genus $g\geq 2$. 
We set $Y=X\times _{U}X$. Let 
$\Delta$ be the image of the diagonal 
morphism $\Delta_{X/U}:X\to X\times _U X$. We set 
$M=p_1^*K_{X/U}+\Delta$ where 
$p_1:Y=X\times _U X \to X$ is the first projection. 
Let $p\in U$ be any closed point. 
Then $M_p=M|_{Y_p}$ is semi-ample for 
every $p\in U$, where 
$\pi:Y\to U$ is the 
natural map and $Y_p=\pi^{-1}(p)$. 
On the other hand, $M$ is not $\pi$-semi-ample. 
\end{ex}

\section{$\ch (k)=0$ vs.~$\ch (k)>0$}\label{sec4}

The following theorem is a special case of the abundance theorem for 
log surfaces. It is a key step toward showing the finite generation of log 
canonical rings (see Corollary 
\ref{cor13}). 

\begin{thm}\label{thm41}
Let $(X, \Delta)$ be a $\mathbb Q$-factorial projective 
log surface such that $\Delta$ is a $\mathbb Q$-divisor. 
Assume that $K_X+\Delta$ is nef and big. 
Then $K_X+\Delta$ is semi-ample. 
\end{thm}

When $\ch (k)=0$, the proof of Theorem \ref{thm41} heavily depends on 
a Kodaira type vanishing theorem 
and it is one of the hardest parts of \cite{fujino}. 
Section 4 of \cite{fujino} is devoted to the proof of Theorem \ref{thm41}. 
On the other hand, when $\ch (k)>0$, 
the proof of Theorem \ref{thm41} is much simpler by Theorem \ref{main-thm}. 
Therefore, in some sense, the minimal model theory of log surfaces is easier to 
treat in positive characteristic. 
In $\ch (k)=0$, it follows from the mixed Hodge theory of compact support 
cohomology groups. 
In $\ch (k)>0$, it uses Frobenius maps (see the proof of Theorem \ref{main-thm}). 
 
\begin{proof}[{Sketch of the proof {\em{($\ch (k)>0$)}}}]
First, we set  
$$
\mathcal E(K_X+\Delta):=\{C\, |\, C\ \text{is a curve on $X$ and}\ C\cdot (K_X+\Delta)=0\}. $$ 
Then $\mathcal E(K_X+\Delta)$ consists of finitely many irreducible curves on $X$ 
by the Hodge index theorem. 
We take an irreducible curve $C\in \mathcal E(K_X+\Delta)$. 
Then $C^2<0$ by the Hodge index theorem. 
If $(K_X+C)\cdot C<0$, 
then $C\simeq \mathbb P^1$ by adjunction and we can contract 
$C$ to a point by Corollary \ref{cor22}. 
Therefore, we may assume that 
$C$ is an irreducible component of $\llcorner \Delta\lrcorner$, 
$C\cap \Supp (\Delta-C)
=\emptyset$, and 
$(K_X+\Delta)\cdot C=0$  for every $C\in \mathcal E(K_X+\Delta)$. 
If $C\simeq \mathbb P^1$ for $C\in \mathcal E(K_X+\Delta)$, 
then it is obvious that $(K_X+\Delta)|_C$ is semi-ample. 
If $C\not\simeq \mathbb P^1$ for 
$C\in \mathcal E(K_X+\Delta)$, then 
we can also check that 
$(K_X+\Delta)|_C$ is semi-ample by adjunction. 
Therefore, by Theorem \ref{main-thm}, we obtain 
that $K_X+\Delta$ is semi-ample. 
\end{proof}

For the details of Theorem \ref{thm41}, see \cite[Section 4]{fujino} and 
\cite[Section 5]{tanaka}. 

\begin{proof}[Sketch of the proof of {\em{Corollary \ref{cor13}} ($\ch (k)>0$)}] 
If $\kappa (X, K_X+\Delta)\leq 1$, then it is obvious that 
$R(X, \Delta)$ is a finitely generated $k$-algebra. 
So, we assume that 
$K_X+\Delta$ is big. 
If $K_X+\Delta$ is not nef, then 
we can find a curve $C$ on $X$ such that $(K_X+\Delta)\cdot C<0$ and $C^2<0$. 
Therefore, 
$(K_X+C)\cdot C<0$. 
By adjunction, $C\simeq \mathbb P^1$. 
By Corollary \ref{cor22}, 
we can contract $C$. 
After finitely many steps, we may assume that $K_X+\Delta$ is nef. 
By Theorem \ref{thm41}, 
$K_X+\Delta$ is semi-ample. 
Thus, $R(X, \Delta)$ is a finitely generated $k$-algebra. 
\end{proof}

\section{$k\ne \overline {\mathbb F}_p$ vs.~$k=\overline {\mathbb F}_p$}
\label{sec4new}

First, we note the following important result. 

\begin{thm}[{see, for example, \cite[Theorem 10.1]{tanaka}}]\label{thm-qfac} 
Let $X$ be a normal surface defined over $\overline {\mathbb F}_p$. 
Then $X$ is $\mathbb Q$-factorial. 
\end{thm}

One of the key results for the minimal model theory of $\mathbb Q$-factorial 
log surfaces is as follows. It plays crucial roles in the 
proof of the non-vanishing theorem 
and the abundance theorem for $\mathbb Q$-factorial log surfaces. 
For details, see \cite{fujino} and \cite{tanaka}. 

\begin{thm}[{cf.~\cite[Lemma 5.2]{fujino} and \cite[Theorem 4.1]{tanaka}}]\label{thm31}
Assume that $k\ne \overline{\mathbb F}_p$. 
Let $X$ be a $\mathbb Q$-factorial projective 
surface and let $f:Y\to X$ be a projective 
birational 
morphism from a smooth projective surface $Y$. Let $p:Y\to C$ be 
a projective surjective morphism onto a projective smooth 
curve $C$ with the genus $g(C)\geq 1$. 
Then every $f$-exceptional curve $E$ on $Y$ is contained in a fiber of 
$p:Y\to C$. 
\end{thm}
\begin{proof}[Sketch of the proof]
By taking suitable blow-ups, we may assume that $E$ is smooth. 
Let $\{E_i\}_{i\in I}$ be the set of all $f$-exceptional 
divisors. 
Suppose that $p(E)=C$. 
We consider the 
subgroup $G$ of $\Pic (E)(k)$ generated by 
$\{\mathcal O_E(E_i)\}_{i\in I}$. 
Since $k\ne \overline {\mathbb F}_p$, 
$(\pi^*\Pic ^0(C))(k)\otimes _{\mathbb Z}\mathbb Q\setminus G\otimes 
_{\mathbb Z}\mathbb Q$ is not 
empty where $\pi=p|_{E}:E\to C$. 
Here, we used the fact that 
the rank of $(\pi^*\Pic ^0(C))(k)$ is infinite 
since $k\ne \overline {\mathbb F}_p$ 
(see \cite[Theorem 10.1]{fj}). 
On the other hand, 
$(\pi^*\Pic^0(C))(k)\otimes_{\mathbb Z}\mathbb Q\subset 
G\otimes _{\mathbb Z}\mathbb Q$ since 
$X$ is $\mathbb Q$-factorial. 
It is a contradiction. 
Therefore, $E$ is in a fiber of $p:Y\to C$. 
\end{proof}

Theorem \ref{thm31} does not hold when $k=\overline {\mathbb F}_p$. 

\begin{ex}
We consider $C=(x^3+y^3+z^3=0)\subset \mathbb P^2=H$, which 
is a hyperplane in $\mathbb P^3$, 
over $k=\overline {\mathbb F}_p$ with $p\ne 3$. 
Let $X$ be the cone over $C$ in $\mathbb P^3$ with the vertex $P$. 
Let $Z\to \mathbb P^3$ be the blow-up at $P$ and let $Y$ be the strict transform of $X$. 
Then $Y$ is a $\mathbb P^1$-bundle over $C$, the singularity of $X$ is not rational, 
$X$ is $\mathbb Q$-factorial (see Theorem \ref{thm-qfac}), and $f:Y\to X$ 
contracts a section of $p:Y\to C$. 
\end{ex}

If $k=\overline {\mathbb F}_p$, then we can easily obtain the 
finite generation of sectional rings. 

\begin{thm}
Assume that $k=\overline {\mathbb F}_p$. Let $X$ be a projective 
surface and let $D$ be a Weil divisor on $X$. 
Then the sectional ring 
$$
R(D)=\bigoplus _{m\geq 0}H^0(X, \mathcal O_X(mD))
$$ 
is a finitely generated $k$-algebra. 
\end{thm}

\begin{proof}[Sketch of the proof] 
As in the proof of Corollary \ref{cor13}, 
we may assume that $D$ is big. 
By contracting curves $C$ with $D\cdot C<0$, 
we may further assume that $D$ is nef and big. 
Then $D$ is semi-ample by Corollary \ref{cor23}. 
Therefore, $R(D)$ is finitely generated. 
\end{proof}

The geometry over $\overline {\mathbb F}_p$ seems to be 
completely different from 
that over $k \ne \overline {\mathbb F}_p$. 
The minimal model theory for log surfaces 
over $k=\overline {\mathbb F}_p$ is discussed in \cite[Part 2]{tanaka}, 
which has a slightly different flavor from that over $k\ne \overline {\mathbb F}_p$. 
 
\section{Indecomposable curves of canonical type}\label{sec-5}

In this section, we discuss a key result for the proof of the abundance theorem 
for $\kappa =0$:~Theorem \ref{thm0}. Note that the abundance theorem for $\kappa =1$ is easy to 
prove and the abundance theorem for $\kappa =2$ has already been treated in Theorem \ref{thm41}. 

\begin{thm}[{cf.~\cite[Theorem 6.2]{fujino} and \cite[Theorem 7.5]{tanaka}}]\label{thm0}
Let $(X, \Delta)$ be a $\mathbb Q$-factorial 
projective log surface such that $\Delta$ is a $\mathbb Q$-divisor. 
Assume that $K_X+\Delta$ is nef and $\kappa (X, K_X+\Delta)=0$. Then 
$K_X+\Delta\sim _{\mathbb Q}0$. 
\end{thm}

Let us recall the definition 
of {\em{indecomposable curves of canonical type}} in the 
sense of Mumford. 

\begin{defn}[Indecomposable curves of canonical type]
Let $X$ be a smooth 
projective surface and let $Y$ be an effective 
divisor on $X$. 
Let $Y=\sum _{i=1}^k n_i Y_i$ be the 
prime decomposition. 
We say that $Y$ is an 
{\em{indecomposable curve of canonical type}} if 
$K_X\cdot Y_i=Y\cdot Y_i=0$ for 
every $i$, $\Supp Y$ is connected, 
and the greatest common divisor of integers $n_1, \cdots, n_k$ is 
equal to one. 
\end{defn}

The following proposition is a key result. 
For a proof, see, for example, \cite[Lemma]{masek} and the proof of 
\cite[Theorem 2.1]{totaro}. See also \cite[Proposition 7.3]{tanaka}.  

\begin{prop}\label{prop1} 
Let $X$ be a smooth projective surface over $k$ and let $Y$ be 
an indecomposable curve of canonical type. 
If $\mathcal O_Y(Y)$ is torsion and $H^1(X, \mathcal O_X)=0$, then 
$Y$ is semi-ample and $\kappa (X, Y)=1$. 
If $\mathcal O_Y(Y)$ is torsion and 
$\ch (k)>0$, 
then $Y$ is always semi-ample and $\kappa (X, Y)=1$ without assuming 
$H^1(X, \mathcal O_X)=0$. 
Therefore, if $k=\overline {\mathbb F}_p$, then 
$Y$ is semi-ample and $\kappa(X, Y)=1$ since $\mathcal O_Y(Y)$ is always torsion. 
\end{prop}

For the details of our proof of the abundance theorem for 
$\kappa =0$, that is, Theorem \ref{thm0}, 
see \cite[Section 7]{tanaka}. 
 
\section{Relative vanishing theorems}\label{sec2} 

The following theorem is a special case of 
\cite[2.2.5 Corollary]{koko} (see also \cite[Theorem 9.4]{kollar-book}). 
Note that 
it holds over any algebraically closed field. 
We also note that the Kodaira vanishing theorem does not 
always hold for surfaces if the characteristic of the base field 
is positive. 

\begin{thm}[Relative vanishing theorem]\label{thm-a}
Let $\varphi:V\to W$ be a proper birational 
morphism from a smooth surface $V$ to a normal 
surface $W$. Let $\mathcal L$ be a line bundle 
on $V$. 
Assume that 
$$
\mathcal L\equiv_{\varphi} K_V+E+N
$$
where $\equiv_{\varphi}$ denotes the relative numerical equivalence, 
$E$ is an effective $\varphi$-exceptional 
$\mathbb R$-divisor 
on $V$ such that $\llcorner E\lrcorner=0$, and $N$ is a $\varphi$-nef 
$\mathbb R$-divisor on $V$. 
Then $R^1\varphi_*\mathcal L=0$. 
\end{thm}

As an application of Theorem \ref{thm-a}, we obtain 
Theorem \ref{thm-b}, whose formulation is suitable for our theory of 
log surfaces. 

\begin{thm}\label{thm-b} 
Let $(X, \Delta)$ be a log surface. 
Let $f:X\to Y$ be a proper birational 
morphism onto a normal surface $Y$. 
Assume that one of the following conditions holds. 
\begin{itemize}
\item[(1)] $-(K_X+\Delta)$ is $f$-ample. 
\item[(2)] $-(K_X+\Delta)$ is $f$-nef and $\llcorner \Delta\lrcorner=0$. 
\end{itemize}
Then $R^1f_*\mathcal O_X=0$. 
\end{thm}
\begin{proof}
Without loss of generality, we may assume that 
$Y$ is affine. 
When $-(K_X+\Delta)$ is $f$-ample, by perturbing 
the coefficients of $\Delta$, we may assume that 
$\llcorner \Delta\lrcorner=0$. 
More precisely, 
let $H$ be an $f$-ample 
Cartier divisor on $X$. 
Then we can find an effective $\mathbb R$-divisor 
$\Delta'$ on $X$ such that 
$\llcorner \Delta'\lrcorner =0$, $\Delta'\sim _{\mathbb R}
\Delta+\varepsilon H$ for a sufficiently small positive 
real number $\varepsilon$, and $-(K_X+\Delta')$ is $f$-ample. 
Let $\varphi:Z\to X$ be the minimal resolution of $X$. 
We set $K_Z+\Delta_Z=\varphi^*(K_X+\Delta)$. 
Note that $\Delta_Z$ is effective. 
Then we have 
$$
-\llcorner \Delta_Z\lrcorner =K_Z+\{\Delta_Z\}-\varphi^*(K_X+\Delta). $$
By Theorem \ref{thm-a}, 
$$
R^1\varphi_*\mathcal O_Z(-\llcorner \Delta_Z\lrcorner)=R^1(f\circ \varphi)_*
\mathcal O_Z(-\llcorner \Delta _Z \lrcorner)=0.$$  
We note that we can write 
$
\{\Delta_Z\}=E+M
$
where $E$ is a $\varphi$-exceptional (resp.~$(f\circ \varphi)$-exceptional) 
effective $\mathbb R$-divisor with $\llcorner E\lrcorner=0$ 
and $M$ is an effective $\mathbb R$-divisor 
such that every irreducible component of $M$ is 
not $\varphi$-exceptional (resp.~$(f\circ \varphi)$-exceptional). 
In this case, $M$ is $\varphi$-nef (resp.~$(f\circ \varphi)$-nef). 
Since 
$$
0\to \mathcal O_Z(-\llcorner \Delta_Z\lrcorner)\to \mathcal O_Z
\to \mathcal O_{\llcorner \Delta_Z\lrcorner}\to 0, $$ 
we obtain 
$$
0\to \varphi_*\mathcal O_Z(-\llcorner \Delta_Z\lrcorner)\to \mathcal O_X\to 
\varphi_*\mathcal O_{\llcorner \Delta_Z\lrcorner}\to 0. 
$$ 
Since $\llcorner \Delta\lrcorner =0$, $\llcorner \Delta_Z\lrcorner$ is $\varphi$-exceptional. 
Therefore, $\varphi_*\mathcal O_{\llcorner \Delta_Z\lrcorner}$ is a skyscraper sheaf on $X$. 
Thus, we obtain 
$$
\cdots \to R^1f_*(\varphi _*\mathcal O_Z(-\llcorner \Delta_Z\lrcorner))\to R^1f_*\mathcal O_X\to 0. $$ 
Since $$R^1f_*(\varphi_*\mathcal O_Z(-\llcorner \Delta_Z\lrcorner))\subset 
R^1(f\circ \varphi)_*\mathcal O_Z(-\llcorner \Delta_Z\lrcorner)=0, $$  
we obtain $R^1f_*\mathcal O_X=0$. 
\end{proof}

We close this section with the following important results. 
For definitions, see \cite[Notation 4.1]{km}. 

\begin{prop}[{cf.~\cite[Proposition 4.11]{km} and 
\cite[Proposition 3.5]{fujino}}]\label{prop53} 
Let $X$ be an algebraic surface defined over an algebraically 
closed field $k$ of arbitrary characteristic. 
\begin{enumerate}
\item Let $(X, \Delta)$ be a numerically dlt pair. 
Then every Weil divisor on $X$ is $\mathbb Q$-Cartier, 
that is, $X$ is $\mathbb Q$-factorial. 
\item Let $(X, \Delta)$ be a numerically lc pair. 
Then it is lc. 
\end{enumerate}
\end{prop}
The proof given in \cite{fujino} 
works over any algebraically closed field once we adopt 
Artin's lemmas (see \cite[Lemmas 3.3 and 3.4]{bad}) 
instead of \cite[Theorem 4.13]{km} since 
the relative Kawamata--Viehweg vanishing 
theorem holds by Theorem \ref{thm-a}. 

\begin{thm}[{cf.~\cite[Theorem 4.12]{km}}]\label{rat}  
Let $X$ be an algebraic surface defined over 
an algebraically closed filed $k$ of arbitrary characteristic. 
Assume that $(X, \Delta)$ is numerically dlt. 
Then $X$ has only rational singularities. 
\end{thm} 

Theorem \ref{rat} follows from Theorem \ref{thm-b} (2). 

\begin{rem}
The proof of Proposition \ref{prop53} uses the classification of 
the dual graphs of the exceptional curves of log canonical surface 
singularities. In the framework of \cite{tanaka}, 
we do not need Proposition \ref{prop53} or the classification of 
log canonical surface singularities even for 
the minimal model theory of log canonical surfaces (see \cite[Part 3]{tanaka}). 
So, we are released from the classification of log canonical surface singularities 
when we discuss the minimal model theory of log surfaces. 
\end{rem}



\begin{thebibliography}{CMM}
\bibitem[A]{artin} 
M.~Artin, 
Some numerical criteria for contractability of curves on algebraic surfaces, 
Amer. J. Math. {\textbf{84}} (1962), 485--496.

\bibitem[B]{bad}
L.~B\u adescu, {\em{Algebraic surfaces}}, 
Translated from the 1981 Romanian original by Vladimir 
Ma\c sek and revised by the author. Universitext. Springer-Verlag, 
New York, 2001.  

\bibitem[CMM]{cmm} 
P.~Cascini, J.~M\textsuperscript{c}Kernan, M.~Musta\c t\u a, 
The augmented base locus in positive characteristic, preprint (2011). 

\bibitem[FJ]{fj} 
G.~Frey, M.~Jarden, 
Approximation theory and the rank of abelian varieties over 
large algebraic fields, Proc. London Math. Soc. (3) {\textbf{28}} (1974), 112--128.

\bibitem[Fn]{fujino} 
O.~Fujino, 
Minimal model theory for log surfaces, 
Publ. Res. Inst. Math. Sci. {\textbf{48}} (2011), no. 2, 339--371. 

\bibitem[Ft1]{fujita-lnm} 
T.~Fujita, Vanishing theorems for semipositive line bundles, 
{\em{Algebraic geometry}} (Tokyo/Kyoto, 1982), 519--528, 
Lecture Notes in Math., {\textbf{1016}}, Springer, Berlin, 1983.

\bibitem[Ft2]{fujita2} 
T.~Fujita, 
Semipositive line bundles, 
J. Fac. Sci. Univ. Tokyo Sect. IA Math. {\textbf{30}} (1983), no. 2, 353--378. 

\bibitem[Ft3]{fujita3} 
T.~Fujita, 
Fractionally logarithmic canonical rings of algebraic surfaces,
J. Fac. Sci. Univ. Tokyo Sect. IA Math. {\textbf{30}} (1984), no. 3, 685--696. 

\bibitem[Ke]{keel}
S.~Keel, 
Basepoint freeness for nef 
and big line bundles in positive characteristic, 
Ann. of Math. (2) {\textbf{149}} (1999), no. 1, 253--286.

\bibitem[Kn]{knu} 
D.~Knutson, {\em{Algebraic spaces}}, 
Lecture Notes in Mathematics, Vol. {\textbf{203}}. Springer-Verlag, Berlin-New York, 1971.

\bibitem[Ko]{kollar-book} 
J.~Koll\'ar, {\em{Singularities of the Minimal Model Program}}, 
preprint (2011). 

\bibitem[KK]{koko} 
J.~Koll\'ar, S.~Kov\'acs, Birational 
geometry of log surfaces, preprint. 

\bibitem[KM]{km} 
J.~Koll\'ar, S.~Mori, {\em Birational geometry of 
algebraic varieties,} Cambridge Tracts in Mathematics, Vol. {\textbf {134}}, 
1998.

\bibitem[M]{masek} 
V.~Ma\c sek, Kodaira--Iitaka and numerical dimensions of 
algebraic surfaces over the algebraic closure of a finite field, 
Rev. Roumaine Math. Pures Appl. {\textbf{38}} (1993), no. 7-8, 679--685. 

\bibitem[Ta1]{tanaka} 
H.~Tanaka, Minimal models and abundance for positive characteristic 
log surfaces, preprint (2011).  

\bibitem[Ta2]{tanaka2} 
H.~Tanaka, X-method for klt surfaces in positive characteristic, preprint (2012). 

\bibitem[To]{totaro} B.~Totaro, 
Moving codimension-one subvarieties over finite fields, 
Amer. J. Math. {\textbf{131}} (2009), no. 6, 1815--1833. 
\end{thebibliography}
\end{document}